\documentclass[11pt]{amsart}
\usepackage[margin=1 in]{geometry}
\usepackage{amsmath}
\usepackage{amssymb}
\usepackage{amsthm}
\usepackage{xcolor}
\usepackage{graphicx}
\usepackage{hyperref}
\hypersetup{
  bookmarksnumbered=true,
  bookmarks=true,
  colorlinks=true,
  linkcolor=blue,
  citecolor=blue,
  filecolor=blue,
  menucolor=blue,
  pagecolor=blue,
  urlcolor=blue,
  pdfnewwindow=true,
  pdfstartview=FitBH}
\usepackage[capitalise]{cleveref}
\usepackage[space]{grffile}
\usepackage{float}
\usepackage{dsfont}

\usepackage[english]{babel}
\newtheorem{theorem}{Theorem}[section]
\newtheorem{proposition}[theorem]{Proposition}
\newtheorem{lemma}[theorem]{Lemma}

\newtheorem{corollary}[theorem]{Corollary}
\theoremstyle{definition}
\newtheorem{definition}[theorem]{Definition}

\newtheorem{remark}[theorem]{Remark}
\theoremstyle{remark}

\newcommand{\R}{\mathbb{R}}

\newcommand{\Q}{\mathbb{Q}}

\newcommand{\Z}{\mathbb{Z}}
\newcommand{\Tr}{\operatorname{Tr}}
\newcommand{\sdfrac}[2]{\mbox{\small$\displaystyle\frac{#1}{#2}$}}
\DeclareMathOperator{\distance}{\mathfrak{d}}

\newcommand{\lp}{\left(}
\newcommand{\rp}{\right)}

\title{On a Special Metric in Cyclotomic Fields}

\author[K. Saettone, A. Zaharescu, Z. Zhang]
{Katerina Saettone, Alexandru Zaharescu, Zhuo Zhang}

\address{
Katerina Saettone: Department of Mathematics,
University of Illinois Urbana-Champaign,
Altgeld Hall, 1409 W. Green Street,
Urbana, IL, 61801, USA}
\email{kas18@illinois.edu}

\address{
Alexandru Zaharescu: Department of Mathematics,
University of Illinois Urbana-Champaign,
Altgeld Hall, 1409 W. Green Street,
Urbana, IL, 61801, USA and Simion Stoilow Institute of Mathematics of the Romanian Academy, 
P. O. Box 1-764, RO-014700 Bucharest, Romania}
\email{zaharesc@illinois.edu}  
\address{
Zhuo Zhang: Department of Mathematics,
University of Illinois Urbana-Champaign,
Altgeld Hall, 1409 W. Green Street, 
Urbana, IL, 61801, USA}
\email{zhuoz4@illinois.edu}

\makeatletter
\@namedef{subjclassname@2020}{%
  \textup{2020} Mathematics Subject Classification}
\makeatother

\begin{document}

\begin{abstract}
	Let $p$ be an odd prime, and let $\omega$ be a primitive $p$th root of unity. In this paper, we introduce a metric on the cyclotomic field $K=\Q(\omega)$. We prove that this metric has several remarkable properties, such as invariance under the action of the Galois group. Furthermore, we show that points in the ring of integers $\mathcal{O}_K$ behave in a highly uniform way under this metric. More specifically, we prove that for a certain hypercube in $\mathcal{O}_K$ centered at the origin, almost all pairs of points in the cube are almost equi-distanced from each other, when $p$ and $N$ are large enough. When suitably normalized, this distance is exactly $1/\sqrt{6}$.
\end{abstract}

\maketitle

\section{Introduction}

Cyclotomic fields play an essential role in algebra and number theory, 
particularly in understanding the behaviour of prime numbers, and the solutions to Diophantine equations.
In this paper, we uncover properties of cyclotomic fields equipped with a special metric, which we study from both algebraic and probabilistic standpoints. 

Let $p$ be an odd prime and let $\omega$ be a primitive $p$-th root of unity.
The extension of $\Q$ generated by $\omega$ in the field of complex numbers is the $p$-th 
cyclotomic field $K=\Q(\omega)$. We shall denote by $\Tr_{K/\Q}$ the \emph{trace map} of the number field $K$ (the precise definition of $\Tr_{K/\Q}$ will be reviewed in \cref{Definition of the metric}).


For $\alpha\in K$, we denote by $v_\alpha$ the vector in $\Q^{p-1}$ whose $j$th component is $\Tr_{K/\Q}(\alpha\omega^j)$, for $1\le j\le p-1$. In this paper, we define $d(\alpha,\beta)$, the \emph{distance} between $\alpha$ and $\beta$ in $K$, as the Euclidean distance between the vectors $v_{\alpha}$ and $v_{\beta}$ 
in $\Q^{p-1}$. We shall show that $d$ is a metric on $K$, where positive-definiteness is the only nontrivial property. Note that $d$ is canonically defined and is independent of the choice of $\omega$.

We aim to investigate this metric $d$ from several perspectives. In \cref{properties}, we show that $d$ has certain nice properties that are related to the algebraic and number-theoretic structure of $\Q(\omega)$. For instance, the metric $d$ is invariant under the action of the Galois group $G=\operatorname{Gal}(K/\Q)$. In turn, this gives us an analogy of \emph{Krasner's lemma} within the context of cyclotomic fields equipped with the metric $d$. 

In \Cref{Computing the metric in coordinates}, we derive an explicit formula for the metric in terms of the coordinates under the canonical basis $\{\omega,\ldots,\omega^{p-1}\}$ of $K$.

In the rest of the paper, we build on the ideas of \cite{ACZ2024} and \cite{ACZ2023} to study the metric $d$ from a statistical point of view. More specifically, for a positive integer $N$, we denote by $B(p,N)$ the 
symmetric \emph{box of cyclotomic lattice points}:
\begin{equation*}
    B(p,N):=\big\{a_1\omega+\cdots+a_{p-1}\omega^{p-1}
    : a_1,\dots,a_{p-1}\in [-N,N]\cap\Z
    \big\},
\end{equation*}
which lies in the ring of integers $\mathcal{O}_K$. In \cref{A normalized distance}, we normalize the metric $d$ so that the diameter of $B(p,N)$ is exactly 1 in the sense of metric spaces, i.e., the points furthest apart in $B(p,N)$ 
are at a distance of exactly $1$ from each other.
This gives us a scaled distance, denoted by $\distance_{p,N}(\alpha,\beta)$, 
which serves as a unitary means of comparing 
the spacing of points in different hypercubes $B(p,N)$, as $p$ and $N$ vary. Our main theorem states that points in $B(p,N)$ are almost equi-distanced from each other in the following sense.

\begin{theorem}\label{Main theorem}
	For any $\varepsilon>0$, there exists an absolute and effectively computable constant $A(\varepsilon)$ such that if $N,p>A(\varepsilon)$, then
	$$
	\frac{1}{\#B(p,N)^{2}}\#\left\{(\alpha,\beta)\in B(p,N)\times B(p,N):\left|\mathfrak{d}_{p,N}(\alpha,\beta)-\frac{1}{\sqrt{6}}\right|>\varepsilon\right\}<\varepsilon.
	$$
\end{theorem}

\Cref{Main theorem} reveals a surprising uniformity in the spacing of points among the high-dimensional lattice points in $K$. It provides insight into a certain ``statistical regularity'' in the geometric properties of cyclotomic fields when viewed through the lens of this particular metric. \Cref{Main theorem} will follow from \Cref{Quantitative Main Theorem}, which is an explicit quantitative version that we shall prove in \cref{equi-distanced}. Our methods rely on calculating the various \emph{moments} of distances between points in $B(p,N)$.

\section{Notations and definition of the metric}
\label{Definition of the metric}

\subsection{Notations and setup}
In this subsection, we set up some notations and recall some preliminary facts from algebraic number theory that will be needed in the later discussions. More details can be found in \cite{Mar2018}, \cite{Neukirch}, and \cite{Was2012}.

Throughout this paper, let $p$ be an odd prime, and let $\omega$ be a primitive $p$th root of unity, say $\omega=e^{2\pi i/p}$.
Let $K=\Q(\omega)$ be the $p$th cyclotomic field. It is well known that the Galois group $G:=\operatorname{Gal}(K/\Q)$ is isomorphic to the group $(\Z/p\Z)^{\times}$, which is cyclic of order $p-1$.

We denote by $\mathcal{O}_K$ the ring of integers of $K$, that is, the integral closure of $\Z$ in $K$. It is well known that rings of integers have integral bases, and in this case, an integral basis of $\mathcal{O}_K$ is given by $\{\omega,\ldots,\omega^{p-1}\}$. Therefore,
\[
\mathcal{O}_K=\{a_1\omega+\ldots+a_{p-1}\omega^{p-1}:a_1,\ldots,a_{p-1}\in\Z\}.
\]

Many key properties of number fields can be studied via the \emph{trace map}. Since cyclotomic fields are always Galois over $\Q$, the trace map $\Tr_{K/\Q}$ has a simple definition in this case, which is

\begin{equation}\label{eqDefTr}
   \Tr_{K/\Q}(\alpha) = \sum_{\sigma \in \operatorname{Gal}(K/\Q)} \sigma(\alpha), \quad \alpha\in K.
\end{equation}
It can be proved that $\Tr_{K/\Q}(\alpha)\in\Q$ for all $\alpha\in K$. Furthermore, if $\alpha\in \mathcal{O}_K$, then $\Tr_{K/\Q}(\alpha)\in\Z$.

Finally, for complex-valued functions $f$ and $g$, we write $f\ll g$ or $f=O(g)$ to indicate that there exists an absolute an effectively computable constant $C$ such that $|f|\le C|g|$ for all inputs.

\subsection{Definition of the metric}
We now formally define the metric $d$ mentioned in the introduction. The metric is, in fact, induced by a norm on $K$ as a $\Q$-vector space. The norm is defined as follows:

\begin{definition}
    \label{def of norm}
	For any $\alpha\in K$, we define
	$$
	\|\alpha\|=\sqrt{\sum_{j=1}^{p-1}\left(\operatorname{Tr}_{K/\Q}(\alpha\omega^j)\right)^2}=\|v_\alpha\|_E,
	$$
	where $v_\alpha\in\Q^{p-1}$ is the vector whose $j$th component is $\operatorname{Tr}(\alpha\omega^j)$ and $\|\cdot\|_E$ denotes the usual Euclidean norm on $\Q^{p-1}$.
\end{definition}

\begin{definition}\label{def of metric}
	For $\alpha,\beta\in K$, we define their distance $d(\alpha,\beta)$ to be $\|\alpha-\beta\|$. 
\end{definition}

\begin{theorem}
	The function $\|\cdot\|$ defined as in \Cref{def of norm} is a norm on $K$.
\end{theorem}
\begin{proof}
	We verify the three conditions of a norm. The triangle inequality follows immediately from the usual triangle inequality in Euclidean spaces $\Q^{p-1}\subseteq\R^{p-1}$.

    For any $\alpha\in K$ and $\lambda\in\Q$, we need to prove that $\|\lambda\alpha\|=|\lambda|\|\alpha\|$. This follows from the $\Q$-linearity of trace, since it implies that $v_{\lambda\alpha}=\lambda v_\alpha$.

    It remains to prove positive-definiteness. Clearly $\|\alpha\|\ge0$. Suppose $\|\alpha\|=0$. Then $v_\alpha\in\Q^{p-1}$ is the zero vector. Hence,
		$$
		\Tr_{K/\Q}(\alpha\omega^j)=0,
		$$
		for all $j=1,\ldots,p-1$. Suppose $\alpha\ne0$. Then we may write
		$$
		\frac{1}{\alpha}=c_1\omega+\ldots+c_{p-1}\omega^{p-1},\quad\text{where }c_i\in\Q. 
		$$
		Therefore,
		$$
		1=c_1\alpha\omega+\ldots+c_{p-1}\alpha\omega^{p-1}.
		$$
		Taking the trace of both sides, and using the fact that trace is $\Q$-linear, we have
		$$
		p-1=\sum_{j=1}^{p-1}c_j\Tr_{K/\Q}(\alpha\omega^j)=0,
		$$
		which is a contradiction. Hence, $\alpha=0$. \qedhere
		
\end{proof}

It follows that the function $d$ is indeed a metric on $K$. The distance defined in this manner closely resembles the Euclidean distance in vector spaces but also has properties that are well-suited to the study of cyclotomic fields. This will be further explored in \cref{properties}.

We remark that the norm in \cref{def of norm} must be distinguished from the usual norm of an algebraic number (say, over a Galois extension), which is defined to be the product of all its Galois conjugates. There also exist several other notions of norms over number fields. For example, one can define the \emph{Siegel norm} of algebraic numbers (see \cite{AFZ2014} and \cite{FA-Siegel} for its construction and some interesting properties; for some questions related to Siegel's trace problem, see \cite{Siegel} and \cite{FZ2009}). In this paper, the word ``norm'' always refers to the norm we just defined, unless stated otherwise.

\section{Properties of the metric}\label{properties}

The metric $d$ on $K=\Q(\omega)$ defined as above is the main object we investigate in this paper. To convince the readers that the metric is a natural object worth studying, we shall first prove a number of remarkable facts about this metric, the most important one of which is the invariance under the action of the Galois group. This is the content of the following proposition.

\subsection{Invariance under the Galois group action}

\begin{proposition}
    \label{invariance under galois action}
	The metric $d$ is invariant under the action of the Galois group $G=\operatorname{Gal}(K/\Q)$. In other words, for any $\sigma\in G$ and $\alpha,\beta\in K$, we have
    $$
    d(\alpha,\beta)=d(\sigma(\alpha),\sigma(\beta)).
    $$
\end{proposition}
\begin{proof}
	It suffices to show that the norm in \cref{def of norm} is invariant under $G$, i.e., $\|\sigma(\alpha)\|=\|\alpha\|$ for all $\alpha\in K$ and $\sigma\in G$. Suppose $\sigma^{-1}(\omega)=\omega^k$, where $1\le k\le p-1$. Then we have
	\begin{align*}
		\|\sigma(\alpha)\|&=\sqrt{\sum_{j=1}^{p-1}\left(\operatorname{Tr}_{K/\Q}(\sigma(\alpha)\omega^j)\right)^2}
		\\
		&=\sqrt{\sum_{j=1}^{p-1}\left(\operatorname{Tr}_{K/\Q}\left(\sigma\left(\alpha\cdot\sigma^{-1}(\omega)^j\right)\right)\right)^2}
		\\
		&=\sqrt{\sum_{j=1}^{p-1}\left(\operatorname{Tr}_{K/\Q}\left(\alpha\cdot\sigma^{-1}(\omega)^j\right)\right)^2}
		\\
		&=\sqrt{\sum_{j=1}^{p-1}\left(\operatorname{Tr}_{K/\Q}\left(\alpha\omega^{kj}\right)\right)^2},
	\end{align*}
 where the third equality follows from the fact that $\Tr_{K/\Q}$ is invariant under $G$.
	Since $k$ must be coprime to $p$, it follows that $\{kj:1\le j\le p-1\}$ is a permutation of $\{j:1\le j\le p-1\}$. Hence, $\|\sigma(\alpha)\|=\|\alpha\|$, as required.
\end{proof}

\subsection{An analogue of Krasner's lemma}

As a consequence of \Cref{invariance under galois action}, we now prove that the metric $d$ has another surprising property, with which we will draw an analogy between the following Krasner's lemma.

\begin{theorem}[Krasner's lemma]
Let $\kappa$ be a complete field with respect to a nonarchimedean valuation and let $\Omega$ be an algebraic closure of $\kappa$. Let $\alpha \in \Omega$ be separable over $\kappa$ and let $\alpha=\alpha_1, \ldots, \alpha_n$ be the conjugates of $\alpha$ over $\kappa$. Suppose that for $\beta \in \Omega$ we have
$$
|\alpha-\beta|<\left|\alpha-\alpha_i\right| \quad \text { for } i=2, \ldots, n,
$$
where $|\cdot|$ denotes the unique extension of the valuation to $\Omega$. Then $\kappa(\alpha) \subseteq \kappa(\beta)$.
\end{theorem}
\begin{proof}
    See Lemma 8.1.6 in \cite{NW2008}
\end{proof}

We now prove the following analogous result.

\begin{theorem}\label{Krasner analogue}
	Let $K=\Q(\omega)$, where $\omega$ is a primitive $p$th root of unity. Let $\alpha$ be an element of $K$ and let $\alpha_1,\ldots,\alpha_n$ be the conjugates of $\alpha$ over $K$, with $\alpha_1=\alpha$. Suppose that for $\beta\in K$ we have
	$$
	d(\alpha,\beta)<\frac{1}{2}\,d(\alpha,\alpha_i)\quad\text{ for }i=2,\ldots,n,
	$$
        where $d$ is the metric in \cref{def of metric}. Then $\Q(\alpha)\subseteq \Q(\beta)$.
\end{theorem}
\begin{proof}
	Let $G=\operatorname{Gal}(K/\Q)$. By Galois theory, $\Q(\alpha)\subseteq \Q(\beta)$ if and only if $\operatorname{Gal}(K/\Q(\beta))\subseteq\operatorname{Gal}(K/\Q(\alpha))$. Since $G$ is a cyclic group of order $p-1$, the preceding condition is equivalent to	
	$\left|\operatorname{Gal}(K/\Q(\beta))\right|$ dividing $\left|\operatorname{Gal}(K/\Q(\alpha))\right|$, which is then equivalent to $[K:\Q(\beta)]$ dividing $[K:\Q(\alpha)]$. By the tower law, this is equivalent to $[\Q(\alpha):\Q]$ dividing $[\Q(\beta):\Q]$.
	
As in the statement, let $\alpha_1,\ldots,\alpha_n$ be the Galois conjugates of $\alpha$ over $K$, with $\alpha_1=\alpha$. Similarly, let $\beta_1,\ldots,\beta_m$ be the Galois conjugates of $\beta$ over $K$, with $\beta_1=\beta$. Since $K/\Q$ is Galois, we have $[\Q(\alpha):\Q]=n$ and $[\Q(\beta):\Q]=m$. Therefore, we need to prove that $n$ divides $m$ under the hypothesis that $d(\alpha,\beta)<\frac{1}{2}d(\alpha,\alpha_i)$ for all $i=2,\ldots,n$.
	
	Let $$
	r=\frac{1}{2}\min_{2\le i\le n}d(\alpha,\alpha_i).
	$$
	Then $d(\alpha,\beta)<r$. For any element $x\in K$, denote by $B(x,r)$ the open ball centered at $x$ with radius $r$ under the metric $d$. Observe that if $\alpha_i,\alpha_j$ are two \emph{distinct} Galois conjugates of $\alpha$, say $\alpha_i=\sigma(\alpha)$ and $\alpha_j=\tau(\alpha)$, where $\sigma, \tau \in \operatorname{Gal}(K/\Q(\alpha))$, then
	$$
	d(\alpha_i,\alpha_j)=d(\sigma(\alpha),\tau(\alpha))=d(\alpha,\sigma^{-1}\tau(\alpha))\ge 2r.
	$$
	It follows that any two distinct conjugates $\alpha_i$ and $\alpha_j$ are at a distance of at least $2r$ from each other. In particular, the open balls $\left\{B(\alpha_i,r):1\le i\le n\right\}$ are pairwise disjoint.

	We claim that for every $\beta_j$ there exists an $\alpha_i$ such that $\beta_j\in B(\alpha_i,r)$. In other words, the balls contain \emph{all} conjugates of $\beta$. Indeed, if $\beta_j=\sigma(\beta)$, then
	$$
	d(\sigma(\alpha),\beta_j)=d(\sigma(\alpha),\sigma(\beta))=d(\alpha,\beta)<r,
	$$
 by \cref{invariance under galois action}.
	Hence, $\beta_j\in B(\sigma(\alpha),r)$.
	
	Furthermore, we claim that for $1\le i\le n,$
    each ball $B(\alpha_i,r)$   contains the same number of conjugates of $\beta$. Indeed, suppose $\alpha_i=\sigma(\alpha)$. Then by \cref{invariance under galois action} again, we have
	$$
 d(\alpha_i,\beta_j)=d(\sigma(\alpha),\beta_j)=d(\alpha,\sigma^{-1}(\beta_j)).
	$$
	Therefore, $\beta_j\in B(\alpha_i,r)$ if and only if $\sigma^{-1}(\beta_j)\in B(\alpha,r)$. Since $\sigma$ is a bijection, this proves that $B(\alpha,r)$ and $B(\alpha_i,r)$ contain the same number of Galois conjugates of $\beta$. This number is nonzero because $\beta\in B(\alpha,r)$. Since the balls are disjoint, we conclude that $n$ divides $m$, as desired.
\end{proof}

\begin{remark}
	The following example illustrates that the constant $\frac{1}{2}$ is optimal, in the sense that any larger constant would make the statement false. Consider $p=3$, $\alpha=\omega$, and $\beta=-\frac{1}{2}$. Then $\alpha$ only has one Galois conjugate other than itself, namely $\omega^2$. A straightforward computation shows that
	$$
	d(\alpha,\beta)=\frac{3}{\sqrt{2}}\quad\text{and}\quad d(\alpha,\omega^2)=3\sqrt{2}.
	$$
	Therefore,
	$$
	d(\alpha,\beta)=\frac{1}{2}d(\alpha,\omega^2),
	$$
	but $\Q(\alpha)$ is not contained in $\Q(\beta)$.
\end{remark}


As a simple consequence, we deduce the following corollary, which is reminiscent of the primitive element theorem in field theory.

\begin{corollary}\label{primitive element theorem analog}
	Let $\alpha,\beta\in K=\Q(\omega)$. Define $\gamma_n=\alpha+\frac{\beta}{n}$. Then $\Q(\gamma_n)=\Q(\alpha,\beta)$ for all sufficiently large $n$.
\end{corollary}
\begin{proof}
	Clearly $\Q(\gamma_n)\subseteq\Q(\alpha,\beta)$, so it suffices to prove the reverse inclusion. Note that
	$$
	d(\alpha,\gamma_n)=\|\alpha-\gamma_n\|=\left\|\frac{\beta}{n}\right\|=\frac{\|\beta\|}{n}.
	$$
	Thus, when $n$ is sufficiently large, we would have $d(\alpha,\gamma_n)<\frac{1}{2}d(\alpha,\sigma(\alpha))$ for all $\sigma\in G$. \Cref{Krasner analogue} implies that $\Q(\alpha)\subseteq\Q(\gamma_n)$. In particular, $\alpha\in\Q(\gamma_n)$, and so $\beta\in\Q(\gamma_n)$, as desired.
\end{proof}

Not only does \Cref{primitive element theorem analog} prove a special case of the primitive element theorem, but it also provides a simple algorithm to find generators of subextensions of $K$.

\section{Computing the metric in coordinates}\label{Computing the metric in coordinates}

In this section, we aim to derive an explicit formula of the metric $d$ in terms of the coordinates of $\alpha\in K$ under the integral basis $\{\omega,\ldots,\omega^{p-1}\}$. We first note that $\Tr_{K/\Q}(1)=p-1$, and $\Tr_{K/\Q}(\omega)=\ldots=\Tr_{K/\Q}(\omega^{p-1})=-1$. Therefore, if
$$
\alpha=a_1\omega+\ldots+a_{p-1}\omega^{p-1},
$$
then
$$
\Tr_{K/\Q}(\alpha)=-(a_1+\ldots+a_{p-1}),
$$
and for $j=1,\ldots,p-1$, we have
$$
\Tr_{K/\Q}(\alpha\omega^j)=-\sum_{\substack{i=1\\i\ne p-j}}^{p-1}a_i+(p-1)a_{p-j}=\Tr_{K/\Q}(\alpha)+pa_{p-j}.
$$
Therefore,
\begin{align*}
	\|\alpha\|^2&=\sum_{j=1}^{p-1}\Tr_{K/\Q}(\alpha\omega^j)^2
	\\
	&=\sum_{j=1}^{p-1}\left(\Tr_{K/\Q}(\alpha)+pa_{p-j}\right)^2
	\\
	&=\sum_{j=1}^{p-1}\left(\Tr_{K/\Q}(\alpha)+pa_{j}\right)^2
	\\
	&=\sum_{j=1}^{p-1}\left(\Tr_{K/\Q}(\alpha)^2+2pa_j\Tr_{K/\Q}(\alpha)+p^2a_j^2\right)
	\\
	&=(p-1)\Tr_{K/\Q}(\alpha)^2+2p\Tr_{K/\Q}(\alpha)\sum_{j=1}^{p-1}a_j+p^2\sum_{j=1}^{p-1}a_j^2
	\\
	&=(p-1)\Tr_{K/\Q}(\alpha)^2-2p\Tr_{K/\Q}(\alpha)^2+p^2\sum_{j=1}^{p-1}a_j^2
	\\
	&=p^2\sum_{j=1}^{p-1}a_j^2-(p+1)\Tr_{K/\Q}(\alpha)^2.
\end{align*}
Hence, we have arrived at the following convenient formula, which we shall frequently use in the later sections:
\begin{lemma}\label{convienient formula}
	Suppose $\alpha=a_1\omega+\ldots+a_{p-1}\omega^{p-1}\in K$. Then
	\begin{equation}
		\label{formula for norm}
    \|\alpha\|^2=p^2\|\alpha\|_E^2-(p+1)\Tr_{K/\Q}(\alpha)^2,
	\end{equation}
    where $\|\alpha\|_E$ the Euclidean norm of $\alpha$, i.e., $\|\alpha\|_E^2=\sum_{i=1}^{p-1}a_i^2$.
\end{lemma}

Also note that by the Cauchy-Schwarz inequality,
$$
\Tr_{K/\Q}(\alpha)^2=\left(\sum_{j=1}^{p-1}a_j\right)^2\le(p-1)\sum_{j=1}^{p-1}a_j^2=(p-1)\|\alpha\|_E^2,
$$
so we conclude that
$$
\|\alpha\|^2\ge \left(p^2-(p+1)(p-1)\right)\|\alpha\|_E^2=\|\alpha\|_E^2.
$$
In other words, the norm of $\alpha$ is always larger than or equal to the Euclidean norm of $\alpha$.

\section{The normalized distance}\label{A normalized distance}

Let $B(p,N)$ be the hypercube
$$
B(p,N):=\left\{a_1\omega+\ldots+a_{p-1}\omega^{p-1}:a_i\in\Z\cap[-N,N]\right\}\subset \mathcal{O}_K.
$$
Then $B(p,N)$ contains $(2N+1)^{p-1}$ points in total. In this section, we introduce a normalized distance on $B(p,N)$. In order to do so, we shall need to compute the diameter of the hypercube $B(p,N$). This is done in the following lemma.

\begin{lemma}
	\label{diameter}
	The diameter of $B(p,N)$, i.e., the maximum distance between two points in $B(p,N)$, is exactly
	$$
	\operatorname{diam} B(p,N)=2Np\sqrt{p-1},
	$$
	which is achieved by the following pairs of points
 $$
	\alpha=\sum_{i=1}^{p-1}N(-\omega)^{i-1}=N\omega-N\omega^2+\ldots+N\omega^{p-2}-N\omega^{p-1}\quad\text{and}\quad\beta=-\alpha.
	$$
\end{lemma}
\begin{proof}
	It suffices to maximize equation (\ref{formula for norm}) for $\alpha-\beta$, where $\alpha,\beta\in B(p,N)$. Note that
	$$
	\alpha-\beta=2\alpha=2N\omega-2N\omega^2-\ldots+2N\omega^{p-2}-2N\omega^{p-1}.
	$$
	It is easy to see that choosing such $\alpha$ and $\beta$ would simultaneously maximize the Euclidean norm $\|\alpha-\beta\|_E^2$ and minimize the trace term $(\Tr_{K/\Q}(\alpha-\beta))^2$, because in this case $\Tr_{K/\Q}(\alpha-\beta)=0$. Therefore, the maximum distance must be achieved by this pair. It follows from \Cref{convienient formula} that
 \[
 (\operatorname{diam}B(p,N))^2=\|\alpha-\beta\|^2=p^2(p-1)(2N)^2,
 \]
 as required.
\end{proof}

\begin{definition}\label{normalizedistance}
	For $\alpha,\beta\in B(p,N)$, we define the \emph{normalized distance} of $\alpha$ and $\beta$ in the cube by
	$$
	\mathfrak{d}_{p,N}(\alpha,\beta)=\frac{d(\alpha,\beta)}{2Np\sqrt{p-1}}.
	$$
\end{definition}

If we normalize the metric in this way, then the diameter of the hypercube $B(p,N)$ is exactly 1. This normalized distance is not only more aesthetically appealing but also very useful in comparing the distribution of points in different hypercubes $B(p,N)$, as $p$ and $N$ vary.

\section{Almost all points in $B(p,N)$ are almost equi-distanced}

\label{equi-distanced}

In this section we show that, in an appropriate sense, almost all points in $B(p,N)$ are ``equi-distanced'' from each other in the sense of \Cref{Main theorem}.
Our proof replies on the explicit calculations of the second and fourth moments of the distances, which we define below.

\begin{definition}
    Fix $p,N$, and let $k$ be a positive integer. We define the \emph{$k$th moment of distances between points in $B(p,N)$} to be the following averaged sum:
    \[
    M_k(p,N):=\frac{1}{\#B(p,N)^2}\sum_{\alpha\in B(p,N)}\sum_{\beta\in B(p,N)}d(\alpha,\beta)^k.
    \]
\end{definition}

\subsection{Computation of the second moment}
Now, we evaluate the second moment of the distances in the following lemmas.

\begin{lemma}\label{sum of powers}
For integers $r\ge 0$ and $N\ge 1$, consider the sum of powers
\begin{equation*}
   S_{r}(N):= \sum_{-N\le a\le N} a^r.
\end{equation*}
Then we have:
\begin{align*}
  \begin{split}
      S_{r}(N) &=0, \quad\text{ if $r$ is odd,}\\  
      S_{2}(N) &= \sdfrac{1}{3}N(N+1)(2N+1),\\
      S_{4}(N) &= \sdfrac{1}{15}N(N+1)(2N+1)(3N^2+3N-1).
  \end{split}
\end{align*}
\end{lemma}
\begin{proof}
    When $r$ is odd, the sum is zero because $a^r+(-a)^r=0$. When $r$ is even, this follows from the well-known Faulhaber's formula of sums of powers (see \cite{Faulhaber} for example).
\end{proof}

\begin{lemma}
	\label{mean square}
	The second moment of distances between points in $B(p,N)$ is given by
 	\begin{align*}
 	    M_2(p,N)&=\frac{2}{3}(p^3-2p^2+1)N(N+1)
      \\
      &=\frac{2}{3}p^3N^2+O(p^2N^2+p^3N).
 	\end{align*}
\end{lemma}
\begin{proof}
	By \Cref{convienient formula}, we have
	\begin{align*}
		&\sum_{\alpha\in B(p,N)}\sum_{\beta\in B(p,N)}d(\alpha,\beta)^2
		\\
		&=\sum_{-N\le a_1,b_1\le N}\cdots \sum_{-N\le a_{p-1},b_{p-1}\le N}\left(p^2\sum_{i=1}^{p-1}(a_i-b_i)^2-(p+1)\sum_{i=1}^{p-1}\sum_{j=1}^{p-1}(a_i-b_i)(a_j-b_j)\right).
	\end{align*}
	We break this sum into two pieces by linearity. The first piece equals
	\begin{equation}
		\label{first piece}
		p^2\sum_{i=1}^{p-1}\sum_{-N\le a_1,b_1\le N}\cdots \sum_{-N\le a_{p-1},b_{p-1}\le N}(a_i-b_i)^2.
	\end{equation}
	The second piece equals
	$$
	(p+1)\sum_{i=1}^{p-1}\sum_{j=1}^{p-1}\sum_{-N\le a_1,b_1\le N}\cdots \sum_{-N\le a_{p-1},b_{p-1}\le N}(a_i-b_i)(a_j-b_j).
	$$
	We now simplify the second piece. If $i\ne j$, then the terms $a_i-b_i$ and $a_j-b_j$ are independent, in which case the sum is zero because
	\begin{equation}\label{independent sum vanishes}
	    \sum_{-N\le a_i, b_i\le N}(a_i-b_i)=0.
	\end{equation}
	If $i=j$, then $(a_i-b_i)(a_j-b_j)=(a_i-b_i)^2$, in which case the sum becomes
	$$
	(p+1)\sum_{i=1}^{p-1}\sum_{-N\le a_1,b_1\le N}\cdots \sum_{-N\le a_{p-1},b_{p-1}\le N}(a_i-b_i)^2,
	$$
	which is exactly the same as (\ref{first piece}), up to a difference in the coefficient. It follows that
    \begin{align*}
        &\sum_{\alpha\in B(p,N)}\sum_{\beta\in B(p,N)}d(\alpha,\beta)^2
        \\
        &=(p^2-p-1)\sum_{i=1}^{p-1}\sum_{-N\le a_1,b_1\le N}\cdots \sum_{-N\le a_{p-1},b_{p-1}\le N}(a_i-b_i)^2
        \\
        &=(p^2-p-1)(p-1)(2N+1)^{2p-4}\sum_{-N\le a_1,\le b_1\le N}(a_1^2-2a_1b_1+b_1^2).
    \end{align*}
    Again, we break the above sum by linearity, and noting that
	\begin{equation}\label{prev_first}
	    \sum_{-N\le a_1,b_1\le N}a_1b_1=\bigg(\sum_{-N\le a_1\le N}a_1\bigg)\bigg(\sum_{-N\le b_1\le N}b_1\bigg)=0,
	\end{equation}
        and
        \begin{equation}\label{prev_second}
            \sum_{-N\le a_1,b_1\le N}(a_1^2+b_1^2)=2(2N+1)\sum_{-N\le a_i\le N}a_i^2=2(2N+1)S_2(N),
        \end{equation}
        where the value of $S_2(N)$ is computed in \Cref{sum of powers}. Therefore, we obtain
	\begin{align*}
         \sum_{\alpha\in B(p,N)}\sum_{\beta\in B(p,N)}d(\alpha,\beta)^2&=2(p^2-p-1)(p-1)(2N+1)^{2p-3}\cdot\frac{1}{3}N(N+1)(2N+1)
         \\
         &=\frac{2}{3}(p^3-2p^2+1)N(N+1)(2N+1)^{2p-2},
	\end{align*}
	and the result follows from dividing the above quantity by $\#B(p,N)^2=(2N+1)^{2p-2}$.
\end{proof}


We will argue that almost all pairs of points $(\alpha,\beta)\in B(p,N)^2$ are almost $\sqrt{\mu}$ away from each other, where
\begin{equation}
	\label{def of mu}
	\mu=\mu(p,N)=\frac{2}{3}p^3N^2
\end{equation}
is exactly the main term appearing in the expression in \Cref{mean square}. To this end, we shall need to compute the fourth moment $M_4(p,N)$.

\subsection{Computation of the fourth moment}
The following lemma will be used several times in the evaluation of $M_4(p,N)$, so we prove it here explicitly.

\begin{lemma}\label{supplementary sum calculations}
We have
\begin{align}
    &\sum^{p-1}_{i=1}\sum^{p-1}_{j=1}\sum_{-N\le a_1,b_1\le N}\cdots \sum_{-N\le a_{p-1},b_{p-1}\le N}\left(a_{i} - b_{i}\right)^{2} \left(a_{j} - b_{j}\right)^{2} \label{double square sum} \\
    &\quad = \frac{2}{45} N (N+1) (p-1) \left(10 N^2 p+4 N^2+10 N p+4 N-3\right).\notag
\end{align}
\end{lemma}

\begin{proof}
We break \eqref{double square sum} into two pieces according to whether $i$ equals $j$. The $i=j$ piece equals:
\begin{align}\label{rewritten i = j case}
     \left(2N + 1\right)^{2p-4} \left(p -1\right) 
     \sum_{-N \leq a_{1}, b_{1} \leq N}
    \left( a_{1} - b_{1} \right)^{4}.
\end{align}
Since $(a_{1} - b_{1})^4=(a_1^4+b_1^4)+4(a_1^3b+a_1b_1^3)+6a_1^2b_1^2$,  we can further rewrite the sum in \eqref{rewritten i = j case} as
\begin{align*}
    2(2N+1)\sum_{-N \leq a_{1} \leq N} a^{4}_{1} + 6 \bigg( \sum_{-N \leq a_{1} \leq N} a^{2}_{1} \bigg)^{2},
\end{align*}
since all term with odd powers vanish. The above quantity can be computed directly using \Cref{sum of powers}.

On the other hand, the $i\ne j$ piece of \eqref{double square sum} equals
\begin{equation}\label{sum of square square}
    \lp 2N + 1 \rp^{2p-6} \left(p-1\right)\left(p - 2\right)\bigg( \sum_{- N \leq a_{1}, b_{1} \leq N} \left( a_{1} - b_{1} \right)^{2} \bigg)^{2}.
\end{equation}
The innermost sum inside the square has been previously calculated in \eqref{prev_first} and \eqref{prev_second}. The result now follows from combining the $i=j$ piece and the $i\ne j$ piece. We omit the details of the tedious calculation.
\end{proof}

 \begin{lemma}
 	\label{mean 4th power}
 	The fourth moment of distances between points in  $B(p,N)$ is given by
 	\begin{align*}
	    M_4(p,N)&=\frac{2}{45} N (N+1) (p-1) ((2 N^2+2 N)
         \left(5 p^5-8 p^4+p^3+8 p^2-21 p-18\right)-3
         (p^2-p-1)^2)
         \\
         &=\frac{4}{9}p^6N^4+O(p^5N^4+p^6N^3).
 	\end{align*}
 \end{lemma}
\begin{proof}
By \Cref{convienient formula}, we need to compute
\begin{align*}
    \sum_{\alpha \in B(p,N)}\sum_{\beta \in B(p,N)} d(\alpha, \beta)^{4}  &= \sum_{-N\le a_1,b_1\le N}\cdots \sum_{-N\le a_{p-1},b_{p-1}\le N}\Bigg( p^{4} \bigg(\sum^{p-1}_{j=1} \left(a_{j} - b_{j}\right)^{2} \bigg)^{2} \\ &- 2p^{2}\big(p+1\big) \bigg( \sum^{p-1}_{i = 1} \lp a_{i} - b_{i} \rp^{2} \bigg) \bigg( \sum^{p-1}_{j = 1}\lp a_{j} - b_{j}\rp \bigg)^{2} + \big( p + 1 \big)^{2} \bigg( \sum^{p-1}_{j = 1} \left( a_{j} - b_{j} \right) \bigg)^{4} \Bigg).
\end{align*}
By linearity, similar to the proof of \Cref{mean square}, we break up the sum above into three pieces.
The first piece of the sum equals
\begin{align}\label{first piece of the sum}
   p^{4} \sum^{p-1}_{i=1} \sum^{p-1}_{j=1} \sum_{-N\le a_1,b_1\le N}\cdots \sum_{-N\le a_{p-1},b_{p-1}\le N} \left( a_{i} - b_{i}\right)^{2} \left(a_{j} - b_{j}\right)^{2}.
\end{align}
Observe that this has been calculated in \Cref{supplementary sum calculations}

We now evaluate the second piece of the sum, which is
\begin{align}\label{second piece of the sum}
   2p^{2}\left(p+1\right)\sum_{-N\le a_1,b_1\le N}\cdots \sum_{-N\le a_{p-1},b_{p-1}\le N} \sum^{p-1}_{i=1} \left(a_{i} - b_{i}\right)^{2} \bigg( \sum^{p-1}_{i=1} \left(a_{i} - b_{i}\right) \bigg)^{2}. 
\end{align}
Omitting the coefficient $2p^2(p+1)$, \eqref{second piece of the sum} equals
\begin{align*}
    &\sum^{p-1}_{i=1} \sum^{p-1}_{j=1} \sum^{p-1}_{k=1} \sum_{-N\le a_1,b_1\le N}\cdots \sum_{-N\le a_{p-1},b_{p-1}\le N} \left(a_{i} - b_{i}\right)^{2}\left(a_{j} - b_{j}\right)\left( a_{k} - b_{k}\right) \\
    &= \sum^{p-1}_{i=1} \sum^{p-1}_{j=1} \sum_{-N\le a_1,b_1\le N}\cdots \sum_{-N\le a_{p-1},b_{p-1}\le N} \left(a_{i} - b_{i}\right)^{2}\left(a_{j} - b_{j}\right)^{2},
\end{align*}
because when $j \neq k$, the sum vanishes as in \eqref{independent sum vanishes}. Hence, \eqref{second piece of the sum} is the same as \eqref{first piece of the sum}, up to a constant multiple, so it can also be calculated using \Cref{supplementary sum calculations}.

The third piece of the sum is
\begin{align}\label{third piece of the sum unsimplified}
   &\left(p + 1 \right)^{2} \sum_{-N\le a_1,b_1\le N}\cdots \sum_{-N\le a_{p-1},b_{p-1}\le N} \bigg( \sum^{p-1}_{j = 1} \left( a_{j} - b_{j} \right) \bigg)^{4}
   \\
   &=(p + 1 )^{2}\sum_{i=1}^{p-1}\sum_{j=1}^{p-1}\sum_{k=1}^{p-1}\sum_{l=1}^{p-1} \sum_{-N\le a_1,b_1\le N}\cdots \sum_{-N\le a_{p-1},b_{p-1}\le N}(a_i-b_i)(a_j-b_j)(a_k-b_k)(a_l-b_l)\tag*{}
\end{align}
Depending on the relations between $i,j,k,$ and $l$, the above sum can be split into pieces that correspond to the set of all integer partitions of 4. For example, if $i=j=k\ne l$, then the partition is $4=3+1$; if $i=j\ne k=l$, then the partition is $4=2+2$. Now, observe that if the partition has an odd number in it (which is either 1 or 3 in this case), then the sum must vanish because
\begin{align*}
    \sum_{-N \leq a_{i}, b_{i} \leq N} \left(a_{i} - b_{i} \right) = \sum_{-N \leq a_{i}, b_{i} \leq N} \left(a_{i} - b_{i}\right)^{3} = 0.
\end{align*}
Hence, only the partitions $4=4$ and $4=2+2$ result in nonzero summands. Therefore, \eqref{third piece of the sum unsimplified} equals (omitting the coefficient $(p+1)^2$)
\begin{align*}
    &\sum_{i=1}^{p-1}\sum_{-N\le a_1,b_1\le N}\cdots \sum_{-N\le a_{p-1},b_{p-1}\le N}(a_i-b_i)^4
    \\
    &+\frac{\binom{4}{2}}{2}\sum_{i=1}^{p-1}\sum_{\substack{j=1\\j\ne i}}^{p-1}\sum_{-N\le a_1,b_1\le N}\cdots \sum_{-N\le a_{p-1},b_{p-1}\le N}(a_i-b_i)^2(a_j-b_j)^2
\end{align*}
which can be further simplified to
\begin{align*}
    &(p-1)(2N+1)^{2p-4}\sum_{-N\le a_1,b_1\le N}(a_1-b_1)^4
\\
&+3(p-1)(p-2)(2N+1)^{2p-6}\bigg(\sum_{-N\le a_1,b_1\le N}(a_1-b_1)^2\bigg)^2.
\end{align*}
We recognize that these two smaller sums have above been previously calculated in the two subcases of \Cref{supplementary sum calculations} (see equations \eqref{rewritten i = j case} and \eqref{sum of square square}, respectively). Again, we omit some details of the tedious calculation.

Now, combining these three pieces gives the total sum in the lemma, and dividing the quantity by $\#B(p,N)^2=(2N+1)^{2p-2}$ yields the result.
\end{proof}

\begin{remark}
    We shall never appeal to the first explicit formula of $M_4(p,N)$ in \Cref{mean 4th power}. Rather, the second asymptotic estimate of $M_4(p,N)$ will be much more useful in the following analyses.
\end{remark}

\subsection{Computation of the second moment about the mean} In this subsection, we apply \cref{mean square,mean 4th power} to obtain an estimate of the \emph{second moment of distances about the mean} between points in $B(p,N)$, which is formally defined by
\[
R(p,N):=\frac{1}{\#B(p,N)^2}\sum_{\alpha\in B(p,N)}\sum_{\beta\in B(p,N)}(d(\alpha,\beta)^2-\mu)^2,
\]
where $\mu$ is defined by (\ref{def of mu}). $R(p,N)$ will play a crucial role in the proof of our main theorem, and following lemma establishes an upper bound of this quantity.
\begin{lemma}
	\label{moment about the mean}
	We have
	$$
	R(p,N) \ll p^5N^4+p^6N^3.
	$$
\end{lemma}
\begin{proof}
	Indeed, we have
	\begin{align*}
		\sum_{\alpha,\beta}(d(\alpha,\beta)^2-\mu)^2&=\sum_{\alpha,\beta}d(\alpha,\beta)^4-2\mu\sum_{\alpha,\beta}d(\alpha,\beta)^2+\mu^2(2N+1)^{2p-2}.
	\end{align*}
	By Lemmas \ref{mean square} and \ref{mean 4th power}, we have
	\begin{align*}
		R(p,N)
		&=\frac{4}{9}p^6N^4+O(p^5N^4+p^6N^3)-2\cdot \left(\frac{2}{3}p^3N^2\right)\cdot \left(\frac{2}{3}p^3N^2+O(p^2N^2+p^3N)\right)+\left(\frac{2}{3}p^3N^2\right)^2
		\\
		&\ll p^5N^4+p^6N^3,
	\end{align*}
	where the main terms cancel nicely, leaving us with only the big-O term.
\end{proof}

\subsection{Proof of \Cref{Main theorem}}

Our main result \Cref{Main theorem} now follows immediately from the following quantitative estimate in \Cref{Quantitative Main Theorem}. Note that, instead of normalizing the distance by a factor of $2Np\sqrt{p-1}$, we chose to normalize it by $2Np^{3/2}$ in \Cref{Quantitative Main Theorem}. This choice makes the computations much cleaner, and it will not at all affect the end result since $2Np\sqrt{p-1}$ and $2Np^{3/2}$ are asymptotic as $p\to\infty$.

\begin{theorem}
    \label{Quantitative Main Theorem}
	For any $\varepsilon>0$ and any positive positive integer $N$,
	\begin{equation}\label{main quantity}
	    \frac{1}{\#B(p,N)^2}\#\left\{(\alpha,\beta)\in B(p,N)^2:\left|\frac{d(\alpha,\beta)}{2Np^{3/2}}-\frac{1}{\sqrt{6}}\right|>\varepsilon\right\}\ll\frac{1}{\varepsilon^2}\left(\frac{1}{p}+\frac{1}{N}\right),
	\end{equation}
	where the implied constant is absolute and effectively computable.
\end{theorem}
\begin{proof}
 Multiplying both sides of the required inequality in \eqref{main quantity} by $2Np^{3/2}$ gives us
    \begin{align*}
        \displaystyle\left|d(\alpha,\beta)-\sqrt{\frac{2}{3}}Np^{3/2}\right|>\varepsilon\cdot 2Np^{3/2},
    \end{align*}
    which may be rewritten as
    \begin{align*}
        \displaystyle\left|d(\alpha,\beta)-\sqrt{\mu}\right|>\varepsilon\sqrt{6\mu}.
    \end{align*}
	Also note that
	$$
	\left|d(\alpha,\beta)^2-\mu\right|=\left|d(\alpha,\beta)+\sqrt{\mu}\right|\left|d(\alpha,\beta)-\sqrt{\mu}\right|\ge\sqrt{\mu}\left|d(\alpha,\beta)-\sqrt{\mu}\right|.
	$$
	Therefore, by \Cref{moment about the mean},
	\begin{align*}
		p^5N^4+p^6N^3&\gg\frac{1}{\#B(p,N)^2}\sum_{\alpha,\beta}(d(\alpha,\beta)^2-\mu)^2
		\\
		&\gg\frac{1}{\#B(p,N)^2}\sum_{\substack{\alpha,\beta\\\left|d(\alpha,\beta)-\sqrt{\mu}\right|>\varepsilon\sqrt{6\mu}}}(d(\alpha,\beta)^2-\mu)^2
		\\
		&\gg\frac{\#\left\{(\alpha,\beta)\in B(p,N)^2:\left|d(\alpha,\beta)-\sqrt{\mu}\right|>\varepsilon\sqrt{6\mu}\right\}}{\#B(p,N)^2}(\varepsilon\sqrt{6\mu}\cdot\sqrt{\mu})^2
		\\
		&=\frac{\#\left\{(\alpha,\beta)\in B(p,N)^2:\left|d(\alpha,\beta)/(2Np^{3/2})-1/\sqrt{6}\right|\right\}}{\#B(p,N)^2}6\varepsilon^2\mu^2.
	\end{align*}
	Therefore, dividing both sides by $6\varepsilon^2\mu^2$, we have
	$$
	 \frac{1}{\#B(p,N)^2}\#\left\{(\alpha,\beta)\in B(p,N)^2:\left|\frac{d(\alpha,\beta)}{2Np^{3/2}}-\frac{1}{\sqrt{6}}\right|>\varepsilon\right\}\ll \frac{1}{6\varepsilon^2}\frac{p^5N^4+p^6N^3}{\mu^2}\ll\frac{1}{\varepsilon^2}\left(\frac{1}{p}+\frac{1}{N}\right),
	$$
	where the implied constant is absolute.
\end{proof}

\bibliographystyle{plainurl}
\nocite{*}
\bibliography{refs}

\begin{thebibliography}{10}

\bibitem{ACZ2024}
J.~Anderson, C.~Cobeli, and A.~Zaharescu.
\newblock Counterintuitive patterns on angles and distances between lattice points in high dimensional hypercubes.
\newblock {\em Results Math.}, 79(2):Paper No. 94, 20, 2024.
\newblock \href {https://doi.org/10.1007/s00025-024-02126-2} {\path{doi:10.1007/s00025-024-02126-2}}.

\bibitem{ACZ2023}
J.S. Athreya, C.~Cobeli, and A.~Zaharescu.
\newblock Visibility phenomena in hypercubes.
\newblock {\em Chaos Solitons Fractals}, 175:Paper No. 114024, 12, 2023.
\newblock \href {https://doi.org/10.1016/j.chaos.2023.114024} {\path{doi:10.1016/j.chaos.2023.114024}}.

\bibitem{Faulhaber}
D.E. Knuth.
\newblock Johann {F}aulhaber and sums of powers.
\newblock {\em Math. Comp.}, 61(203):277--294, 1993.
\newblock \href {https://doi.org/10.2307/2152953} {\path{doi:10.2307/2152953}}.

\bibitem{AFZ2014}
A.~Malik, F.~Stan, and A.~Zaharescu.
\newblock The {S}iegel norm, the length function and character values of finite groups.
\newblock {\em Indag. Math. (N.S.)}, 25(3):475--486, 2014.
\newblock \href {https://doi.org/10.1016/j.indag.2013.12.001} {\path{doi:10.1016/j.indag.2013.12.001}}.

\bibitem{Mar2018}
D.A. Marcus.
\newblock {\em Number Fields}.
\newblock Universitext. Springer New York, 2018.
\newblock \href {https://doi.org/10.1007/978-3-319-90233-3} {\path{doi:10.1007/978-3-319-90233-3}}.

\bibitem{Neukirch}
J.~Neukirch and N.~Schappacher.
\newblock {\em Algebraic Number Theory}.
\newblock Grundlehren der mathematischen Wissenschaften. Springer Berlin Heidelberg, 1999.
\newblock \href {https://doi.org/10.1007/978-3-662-03983-0} {\path{doi:10.1007/978-3-662-03983-0}}.

\bibitem{NW2008}
J.~Neukirch, A.~Schmidt, and K.~Wingberg.
\newblock {\em Cohomology of Number Fields}.
\newblock Grundlehren der mathematischen Wissenschaften. Springer Berlin Heidelberg, 2008.
\newblock \href {https://doi.org/10.1007/978-3-540-37889-1} {\path{doi:10.1007/978-3-540-37889-1}}.

\bibitem{Siegel}
C.L. Siegel.
\newblock The trace of totally positive and real algebraic integers.
\newblock {\em Ann. of Math. (2)}, 46:302--312, 1945.
\newblock \href {https://doi.org/10.2307/1969025} {\path{doi:10.2307/1969025}}.

\bibitem{FZ2009}
F.~Stan and A.~Zaharescu.
\newblock Siegel's trace problem and character values of finite groups.
\newblock {\em J. Reine Angew. Math.}, 637:217--234, 2009.
\newblock \href {https://doi.org/10.1515/CRELLE.2009.097} {\path{doi:10.1515/CRELLE.2009.097}}.

\bibitem{FA-Siegel}
Florin Stan and Alexandru Zaharescu.
\newblock The {S}iegel norm of algebraic numbers.
\newblock {\em Bull. Math. Soc. Sci. Math. Roumanie (N.S.)}, 55(103)(1):69--77, 2012.
\newblock URL: \url{https://www.jstor.org/stable/43679241}.

\bibitem{Was2012}
L.C. Washington.
\newblock {\em Introduction to Cyclotomic Fields}.
\newblock Graduate Texts in Mathematics. Springer New York, 1997.
\newblock \href {https://doi.org/10.1007/978-1-4612-1934-7} {\path{doi:10.1007/978-1-4612-1934-7}}.

\end{thebibliography}

\vspace{26pt}

\end{document}